\documentclass[12pt]{amsart}
\usepackage{amsmath,amsthm,amscd,amssymb,latexsym,upref}
\usepackage{versions}

\makeatletter
\def\theequation{\@arabic\c@equation}

\def\d{\mathbb{D}}
\def\c{\mathbb{C}}

\def\be{\begin{equation}}
\def\ee{\end{equation}}

\def\s0{s_0}
\def\p0{p_0}
\let\phi\varphi
\let\epsilon\varepsilon

\newtheorem{theorem}{Theorem}[section]
\newtheorem{corollary}[theorem]{Corollary}

\newtheorem{proposition}[theorem]{Proposition}

\numberwithin{equation}{section}

\newtheorem{definition}[theorem]{Definition}

\newtheorem{fact*}[theorem]{Fact}
\DeclareMathOperator\hol{Hol}

\DeclareMathOperator\re{Re}

\newcommand{\cala}{\mathcal{A}}

\newcommand\e{\mathrm{e}}

\newcommand\cale{\mathcal{E}}

\newcommand{\T}{\mathbb{T}}

\newcommand{\F}{\mathcal{F}}
\newcommand{\D}{\mathbb{D}}
\newcommand{\C}{\mathbb{C}}

\newcommand{\abs}[1]{\left\vert#1\right\vert}

\newcommand{\inv}{^{-1}}
\newcommand\ov[1]{\overline{#1}}

\newcommand{\E}{\mathcal{E}}

\newcommand\la{\lambda}

\newcommand\df{\stackrel{\rm def}{=}}

\newcommand\beq{\begin{equation}}

\newcommand\eeq{\end{equation}}

\newcommand\bbm{\begin{bmatrix}}
\newcommand\ebm{\end{bmatrix}}
\newcommand\bpm{\begin{pmatrix}}
\newcommand\epm{\end{pmatrix}}

\newcommand\Diag{\mathrm{Diag}}
\newcommand\mudiag{\mu_{\mathrm{Diag}}}

\let\phi\varphi

\numberwithin{equation}{section}
\begin{document}
\title[Analytic interpolation into the tetrablock]{Analytic interpolation into the tetrablock and a $\mu$-synthesis problem}
\author{Zinaida  A.\ Lykova, N.\ J.\ Young and Amos  E. \ Ajibo}
\address{School of Mathematics, Statistics and Physics, Newcastle University, Newcastle upon Tyne
 NE\textup{1} \textup{7}RU, U.K.}
\address{N. J. Young is also affiliated to the School of Mathematics, Leeds University,  Leeds LS2 9JT, U.K.}

\date{5th May 2018}
\dedicatory{To Joe Ball in esteem and friendship}

\thanks{ The first and second  authors were partially supported by the UK Engineering and Physical Sciences Research Council grants  EP/N03242X/1.  The third author  was supported by the Government of Nigeria.}

\begin{abstract} 
We give a solvability criterion for a special case of the $\mu$-synthesis problem.  That is, we prove the necessity and sufficiency of
a condition  for the existence of an analytic $2 \times 2$
matrix-valued  function on the disc subject to a bound on the structured singular value
and satisfying a finite set of
interpolation conditions.   To do this we prove 
 a realization theorem for analytic functions from the disc to the
tetrablock.  We also obtain a solvability criterion for the problem of
analytic interpolation from the disc to the tetrablock. 
\end{abstract}



\maketitle
\sloppy
\fussy
\pagenumbering{arabic}
\setcounter{page}{1}

\section{Introduction}\label{newintro}
The application of operators on Hilbert space to complex analysis has enjoyed about a century of fruitful development since G. Pick gave it a spectacular impulse in \cite{Pick}.  For  the second half of the century Joseph Ball been a prolific contributor to the theory.  One important strand in his research has been to broaden the range of application of operator-theoretic methods to ever new areas of function theory.  In this paper we also take a small step in extending established methods to a new type of domain.

Many of the operator-theoretic methods developed over the first half-century work most smoothly for problems involving analytic functions from the open unit disc $\d$ into certain very special domains, such as Cartan domains.  Such is the case for the methods originated by D. Sarason \cite{sar67} and V. M. Adamyan, D. Z. Arov and M. G. Krein \cite{aak}.    Ball has been prominent from an early stage both in introducing new approaches to the classical problems and in extending the class of domains to which the methods apply.  
As examples of new approaches, witness a series of joint papers with J. W. Helton \cite{bh} in the 1980s on the application of a novel Lax-Beurling theorem in $H^2$ spaces over Krein spaces, and a little later, on the exploitation of $J$-inner-outer factorization of operator-valued analytic functions on the disc.   More recently Ball has been active in the extension of operator-theoretic methods to multivariable, time-varying, nonlinear and noncommutative variants of classical problems \cite{ballHorst2,BMV,BallHuaman}.  It remains a worthwhile goal to explore the reach of operator-theoretic methods in function theory, and in this paper we apply them to the {\em tetrablock},  the domain in $\c^3$ defined by
\beq\label{defE}
\mathcal{E}=\{(x_1,x_2,x_3)\in\mathbb{C}^3:1-x_1z-x_2w+x_3zw\neq 0 \text{ for all }z,w\in \ov{\d}\}.
\eeq

Operator-theoretic techniques can be applied to many problems in analysis.  One that is relevant to this paper is the problem of analytic interpolation:

\em  Given distinct points $\la_1,\dots,\la_N$ in a domain $D$ in some $\c^d$ and target points $w_1,\dots,w_N$ in a subset $E$ of some Banach space $X$, determine whether there exists an analytic map $F:D\to E$ such that $F(\la_j)=w_j$ for $j=1,\dots,N$.

\rm
We shall call such a problem a {\em finite interpolation problem for $\hol(D,E)$}, the latter symbol denoting the set of holomorphic maps from $D$ to $E$.

The exemplar for solutions of finite interpolation problems of this type is the above-mentioned theorem of Pick \cite{Pick}, which provides an elegant criterion for the existence of an interpolating function $F$ in the case that $D=\d$ and $E=\overline{\d}$.
Among the many papers in the literature stemming from Pick's theorem, we mention a relevant paper of Ball and Bolotnikov \cite{bb}, which extends Pick's theorem (and related results) to the case that the target set $E$ has the form 
\beq\label{bbset}
 E=\{z\in\c^d: \|P(z)\| \leq 1\}
\eeq
 where $P$ is a matrix-valued polynomial in $d$ variables and $\|\cdot\|$ is the operator norm with respect to the Euclidean norms.  It seems to be difficult to extend established operator-theoretic methods much beyond this class of target sets while preserving the concreteness of Pick's original criterion, although there are extensions that make use of more abstract notions \cite{DM,BallHuaman}.

In this paper we extend methods from \cite{AY4,BLY} to study the analytic interpolation problem in which the domain $D$ is $\d$ and the target set $E$  is the closure $\ov{\cale}$ of the tetrablock  in $\c^3$.  The closed tetrablock, being the closure of $\cale$, is given by
\beq\label{defEbar}
\ov{\cale}=\{(x_1,x_2,x_3)\in\mathbb{C}^3:1-x_1z-x_2w+x_3zw\neq 0  \text{ for all }z,w\in \mathbb{D}\}.
\eeq
$\ov{\cale}$ is not of the form \eqref{bbset}, and so its function theory is not accessible to the results of \cite{bb}.
$\ov{\cale}$ is the closure of a domain which is both inhomogeneous and non-convex, and so it is to be expected that its function theory should be more complicated in some respects than that of the classical domains.
We derive a criterion for solvability (Theorem \ref{tetra_star2}) of the finite interpolation problem for $\hol(\d,\ov{\cale})$, which while concrete and potentially capable of numerical checking, is not verifiable in rational arithmetic (unlike Pick's criterion).   The original motivation for the study of the tetrablock was its connection with a special case of the $\mu$-synthesis problem, which arises in the theory of robust control.  $H^\infty$ control is another topic on which  Ball has published widely (for example, \cite{ballHorst1,ballHorst2}).  We explain the connection between the tetrablock and a case of the $\mu$-synthesis problem in Section \ref{mu} below.

In addition to a criterion for interpolation we also prove two further results in the function theory associated with $\cale$.  The first is a realization formula for analytic maps from $\d$ to $\ov{\cale}$  (Theorem \ref{realn2})  and the second is a solvability criterion for a certain interpolation problem for matrix-valued analytic functions on $\d$  which is a case of the $\mu$-synthesis problem (Theorem \ref{main-theorem}).

\section{The tetrablock}\label{tetrablock}
The complex geometry, function theory and operator theory of $\cale$
have attracted much attention over the past 10 years, for example \cite{edizwontetra,tetrabhatt,BLY,kos,You08,jp,NTT,pal}.
Although $\cale$ was first studied because of its application to a problem in control theory, it has turned out to be interesting to specialists in several complex variables and operator theory.

We shall require some elementary properties of the tetrablock. 
The following function plays an important role \cite{AWY07}. 
\begin{definition}\label{defPsi}
The rational function $\Psi$
 is defined for $(z,x_1,x_2,x_3) \in \C^4$ such that  $x_2z\neq 1$  by 
\[
\Psi(z,x_1,x_2,x_3)=\frac{x_3z-x_1}{x_2z-1}.
\] 
\end{definition}
$\Psi$ is analytic on the complement of the variety
$x_2z=1$ in $\c^4$. Note that, for $x \in \C^3$ such that 
$x_1x_2=x_3$, the function  $\Psi(\cdot,x)$ is constant and equal to $x_1$. 

Numerous characterizations of the closed tetrablock are given in \cite[Theorem 2.4]{AWY07}. Here are four of them.
\begin{proposition}\label{cond_Ebar}
Let $x=(x_1,x_2,x_3)\in\mathbb{C}^3$. The following are equivalent.
\begin{enumerate}
\item $ x\in \ov{\cale}$;
\item $ |\Psi(z,x)|\leq 1$ for all $z\in\mathbb{D}$ and if $x_1x_2=x_3$  then, in addition, $|x_2|\leq 1$;
\item $ |x_2-\overline{x_1}x_3|+|x_1x_2-x_3|\leq1-|x_1|^2$ and if $x_1x_2=x_3$ then, in addition, $|x_2|\leq1$;
\item $ |x_1-\overline{x_2}x_3|+|x_1x_2-x_3|\leq1-|x_2|^2$ and if $x_1x_2=x_3$ then, in addition, $|x_1|\leq1$;
\item there is a $ 2\times2$  matrix $A=[a_{ij}]_{i,j=1}^2$ such that $\|A\|\leq 1 $ and $x=(a_{11},a_{22},\det{A})$.
\end{enumerate}
\end{proposition}

We denote by  $\mathcal{A}(\mathcal{E})$ the algebra of continuous functions
on $\ov{\cale}$ that are analytic on $\mathcal{E}$. 
By  \cite[Theorem 2.9]{AWY07}, $\ov{\cale}$ is polynomially convex, and so  the maximal ideal space of $\cala(\cale)$ is 
$\ov{\cale}$.  Hence the Shilov boundary of $\cala(\cale)$ is a subset of $\ov{\cale}$, called
the distinguished boundary of $\cale$ and denoted by $b\cale$.
The following alternative descriptions of $b\mathcal{E}$ (among others) are given in \cite[Theorem 7.1]{AWY07}.
\begin{proposition}\label{conditionsforbE}
Let $x=(x_1,x_2,x_3)\in\mathbb{C}^3$. The following are equivalent.
\begin{enumerate}
\item $ x\in b\mathcal{E}$;
\item $ x\in \overline{\mathcal{E}}\;\text{ and }\;|x_3|=1$;
\item $ x_1=\overline{x_2}x_3,\; |x_3|=1\;\text{ and }\;|x_2|\leq 1$.
\end{enumerate}
\end{proposition}

By \cite[Corollary 7.2]{AWY07},
$b\mathcal{E}$ is homeomorphic to $\overline{\mathbb{D}}\times\mathbb{T}$.
We denote the unit circle by $\mathbb{T}$.  

\begin{definition} \label{E-in-funct}  
An  {\em $\mathcal{E}$-inner function} is an analytic function $\phi : \D \to \ov{\cale}$ such that the radial limit
\begin{equation}\label{radialE}
\lim_{r \to 1-} \phi(r \lambda) \mbox{ exists and belongs to } b\mathcal{E}
\end{equation}
for almost all $\lambda \in \T$ with respect to Lebesgue measure.
\end{definition}

By Fatou's Theorem, the radial limit (\ref{radialE}) exists for almost all 
 $\lambda \in \T$ with respect to Lebesgue measure. 
Note that, for an $\mathcal{E}$-inner function  $\phi= (\phi_1, \phi_2, \phi_3) : \D \to \ov{\cale}$,  $\phi_3$ is an inner function on $\mathbb{D}$ in the classical sense.

A finite interpolation problem for $\hol(\D,\ov{\cale})$ has a solution if and only if it has a rational $\mathcal{E}$-inner solution -- see \cite[Theorem 8.1]{BLY}. 



\section{Realization formulae and the tetrablock}\label{realization}

A {\em realization formula} for a class of functions is an expression for a general function in the class in terms of operators on Hilbert space.
We shall give a realization formula for the class $\hol(\d,\ov{\cale})$.

The {\em Schur class} of operator-valued or matricial functions (of a given type) is the set of 
analytic operator- or
matrix-valued functions  $F$ on $\D$ bounded by $1$ in norm, that is,
satisfying 
$$
|| F(\lambda) || \leq 1 \qquad \text{ for all } \lambda \in \D,
$$
where $||\cdot||$ denotes the operator norm.  In particular, the space of analytic $2\times 2$ matrix functions $F$ on $\mathbb{D}$ such that $\|F\|\leq1$ for all $\lambda\in\mathbb{D}$ is called the {\em $2\times 2$ Schur class} and is denoted by $\mathcal{S}^{2\times 2}.$

The best-known realization theorem is for the Schur class (see for example \cite[Chapter VI.3]{NF}). 
If, for some Hilbert spaces $H, U$ and $Y$,
\begin{equation}
\label{coeff}
\left[\begin{array}{cc} A&B\\
C&D\end{array}\right] : H \oplus U \rightarrow H \oplus Y
\end{equation}
is a contractive operator,
then, for any $z\in \mathbb{D}$,
$$
\|D+Cz(1-Az)^{-1}B\|\leq 1.
$$
Conversely, any function in the Schur class (of functions from $\D$ to the
space of bounded linear operators from $U$ to $Y$)
has such a representation, in which the block operator matrix
(\ref{coeff}) is a unitary operator from $H \oplus U$ to $H \oplus Y$.

It will be convenient to use some
standard engineering notation.  If $H, U$ and $Y$  are Hilbert
spaces and
$$
\begin{array}{cccc}
A: H\to H, &&& B:U\to H,\\ \\
C: H\to Y, &&& D: U \to Y\end{array}$$
are bounded linear operators, then we define 
$$
\left[\begin{tabular}{c|c}
   $A$ \hspace{.05in}& $B$\\ \hline
$C$ \hspace{.05in}&$D$
\end{tabular}\right] 
$$
to be the operator-valued function
$$
z \mapsto D + Cz(1-Az)^{-1}B: U \to Y
$$
defined for all $z\in\c$ such that $1-Az$ is invertible.

 The following result \cite[Theorem 7.1]{BLY} relates $\hol(\mathbb{D},\ov{\cale})$ to $\mathcal{S}^{2\times 2}$.

\begin{proposition}\label{tetra-intBLY}
Let $x\in \hol(\mathbb{D},\ov{\cale}).$ There exists a unique function 
\[
F=[F_{ij}]_1^2\in\mathcal{S}^{2\times 2}
\] 
such that 
\beq\label{xandF}
x=(F_{11},F_{22},\det F),
\eeq
where $ (\det F)(z)$ denotes the determinant of the matrix $\det F(z)$ for $  z \in \mathbb{D} $,
 and
 \[
\abs{F_{12}}=\abs{F_{21}} \;a.e.\;\text{on}\;\mathbb{T},\; F_{21}\; \text{is either 0 or outer, and }\;F_{21}(0)\geq0.
\]
 Moreover, for all $\mu,\lambda\in\mathbb{D}$ and all $w,z\in\mathbb{C}$ such that 
\[
1-F_{22}(\mu)w\neq0\;\text{and}\;1-F_{22}(\lambda)z\neq 0,
\]  
\begin{align}
1-\overline{\Psi(w,x(\mu))}\Psi(z,x(\lambda))&=
(1-\overline{w}z)\overline{\gamma(\mu,w)}\gamma(\lambda,z)\nonumber\\
                                             &+\eta(\mu,w)^*(I-F(\mu)^*F(\lambda))\eta(\lambda,z),
\end{align}
where 
\begin{equation}
\gamma(\lambda,z):=(1-F_{22}(\lambda)z)^{-1}F_{21}(\lambda)\;\text{and}\;\eta(\lambda,z):=\left[\begin{array}{c}
1\\z\gamma(\lambda,z)
\end{array}\right].
\end{equation}

Conversely, if $F\in \mathcal{S}^{2\times 2}$ then
\[
(F_{11},F_{22},\det F) \in \hol(\d,\ov{\cale}).
\]
\end{proposition}

 Here is an outline of the construction of $F$ for a given $x\in\hol(\d,\ov{\cale})$.  Certainly $\abs{x_1(\lambda)},\;\abs{x_2(\lambda)}\leq 1$ for all $\lambda\in\mathbb{D}$.  If  $x_1x_2=x_3,$ then we may simply define $F$ by
\[
F=\left[\begin{array}{lr}
x_1&0\\0&x_2
\end{array}\right].
\]

In the case that $x_1x_2\neq x_3,$ the $H^\infty$ function $x_1x_2-x_3$ is nonzero and so it has an inner-outer factorization which can be written in the form 
\[
x_1x_2-x_3=\phi \e^C,
\] 
where $\phi$ is inner, $\e^C$ is outer and $\e^C(0)\geq 0.$  Let $F$ be defined by 
\[
F=\left[\begin{array}{lr}
x_1&\phi e^{\frac{1}{2}C}\\e^{\frac{1}{2}C}&x_2
\end{array}\right].
 \]
Then clearly
\[
\det F=x_1x_2-\phi e^C=x_1x_2-x_1x_2+x_3=x_3
\] 
and 
\[
\abs{F_{12}}=\e^{\re{\frac{1}{2}C}}=\abs{F_{21}}\;\text{a.e. on}\;\mathbb{T},\;F_{21}\;\text{is outer, and}\;F_{21}(0)\geq 0.
\]
The crux of the proof in \cite{BLY} is to show that $F$ is in the Schur class.
\vspace*{0.5cm}

We may combine Proposition \ref{tetra-intBLY} with the classical realization formula to obtain a realization for $\hol(\d,\ov{\cale})$.
\begin{theorem}\label{realn2}
A function 
$$
x=(x_1, x_2, x_3): \mathbb{D} \to \C^3
$$
maps $\D$ analytically into $\ov{\cale}$ 
if and only if there exist a Hilbert 
space $H$ and a unitary operator
\begin{equation}
\label{1.51}
\left[\begin{array}{cc} A&B\\ C&D\end{array}\right]: H \oplus
\mathbb{C}^2 \to H \oplus \mathbb{C}^2
\end{equation}
such that, for all $\lambda \in \D$,
\begin{equation}
\label{got1}
x_1(\lambda) = \left[\begin{tabular}{c|c}
   $A$ \hspace{.05in}& $B_1$\\ \hline
$C_1$ \hspace{.05in}&$D_{11}$
\end{tabular}\right] (\lambda)= D_{11} + C_1 \lambda (1-A \lambda)^{-1}B_1,
\end{equation}
\begin{equation}
\label{got2}
x_2(\lambda)=  \left[\begin{tabular}{c|c}
   $A$ \hspace{.05in}& $B_2$\\ \hline
$C_2$ \hspace{.05in}&$D_{22}$
\end{tabular}\right](\lambda) = D_{22} + C_2 \lambda (1-A \lambda)^{-1}B_2,
\end{equation}
and
\begin{equation}
\label{gotp}
x_3(\lambda)= \det \left[\begin{tabular}{c|c}
   A \hspace{.05in}& B\\ \hline
C \hspace{.05in}&D
\end{tabular}\right](\lambda) = \det[ D + C \lambda (1-A \lambda)^{-1}B],
\end{equation}
where  
\[
B=\bbm B_1 & B_2\ebm : \mathbb{C}^2 \to H, \quad
C=\bbm C_1\\C_2\ebm : H\to
\mathbb{C}^2
\mbox{ and }  D = \bbm D_{ij}\ebm_{i,j=1}^2.
\]
\end{theorem}

\begin{proof}Let $x:\mathbb{D}\to \overline{\E}$ be analytic. By Proposition \ref{tetra-intBLY} there exists $F$ in
 the Schur class such that
$$
x = (F_{11},F_{22}, \det F).
$$
By the realization theorem for the Schur class there exist a Hilbert space $H$ and a
unitary operator $\left[\begin{array}{cc} A&B\\
C&D\end{array}\right]$ on $H\oplus \mathbb{C}^2$ such that, for
all $\lambda \in \mathbb{D}$,
\begin{eqnarray*}
F(\lambda) &=& \left[\begin{tabular}{c|c}
   $A$ \hspace{.05in}& $B$\\ \hline
$C$ \hspace{.05in}&$D$
\end{tabular}\right] (\lambda)\\ \\
&=& D + C\lambda(1-A\lambda)^{-1} B\\ \\
&=&\left[\begin{array}{cc}
   D_{11} &D_{12}\\
D_{21}&D_{22}
\end{array}\right] +
\left[\begin{array}{c}
   C_{1} \\
C_{2}
\end{array}\right] \lambda (1-A\lambda)^{-1} [B_1 \ B_2].
\end{eqnarray*}
Thus
\begin{eqnarray*}
F_{11} &=& \left[\begin{tabular}{c|c}
   $A$ \hspace{.05in}& $B_1$\\ \hline
$C_1$ \hspace{.05in}&$D_{11}$
\end{tabular}\right],\\
F_{22}&=& \left[\begin{tabular}{c|c}
   $A$ \hspace{.05in}& $B_2$\\ \hline
$C_2$ \hspace{.05in}&$D_{22}$
\end{tabular}\right],\\
\det F&=& \det \left[\begin{tabular}{c|c}
   A \hspace{.05in}& B\\ \hline
C \hspace{.05in}&D
\end{tabular}\right].
\end{eqnarray*}
and so equations  \eqref{got1}, \eqref{got2} and \eqref{gotp} hold. 

Conversely, let a Hilbert space $H$ and operators $A, B,  C$ and $D$ satisfy displayed formulae  \eqref{1.51} to \eqref{gotp}.  Then $x$ is clearly analytic in $\d$.  We must show that $x(\d)\subset \ov{\cale}$.

Let
$$
\chi= \left[\begin{tabular}{c|c}
   $A$ \hspace{.05in}&$B$\\ \hline
$C$ \hspace{.05in}&$D$
\end{tabular}\right] = [\chi_{ij}],
$$
so that $\|\chi\|_\infty \le 1$ by the realization theorem for the Schur class.
We have 
$$
\chi_{jj} =  \left[\begin{tabular}{c|c}
   $A$ \hspace{.05in}& $B_j$\\ \hline
$C_j$ \hspace{.05in}&$D_{jj}$
\end{tabular}\right],
\qquad j=1,2,
$$
and so
$$
x_1 = \chi_{11}, \quad \quad x_2 = \chi_{22} \qquad \mbox{ and }\quad x_3 = \det  \chi.
$$
Hence, for any $\la\in\d$, 
\[
x(\la)=(\chi_{11}(\la), \chi_{22}(\la),\det\chi(\la))
\]
where $\chi(\la)$ is a $2\times 2$ matrix of norm at most $1$.  Hence, 
 by  Proposition \ref{cond_Ebar}, for any $\la\in\d$,  $x(\la)\in\ov{\cale}$. Thus $x(\d)\subset \ov{\cale}$.
\end{proof}

\section{The finite interpolation problem for $\hol(\d,\ov{\cale})$}\label{NPforE}
In this section we prove the main theorem of the paper, a criterion for the solvability of the problem in the title of the section.

By {\em matricial Nevanlinna-Pick data} we mean a finite set $\lambda_1,
\dots, \lambda_n$ of distinct points in $\mathbb{D}$, where $n\in
\mathbb{N}$, and an equal number of  ``target matrices" $W_1,
\dots, W_n$, of type (say) $m\times p$.  We write these data
\begin{equation}\label{NPdata}
\lambda_k \mapsto W_k,\qquad k=1,\dots, n.
\end{equation}
We say that these data are {\em solvable} if there exists a
function $F$ in the Schur class such
that $F(\lambda_k) = W_k, \  k=1,\dots, n$.  
The {\em Nevanlinna-Pick problem} is to ascertain whether prescribed data are solvable.
By the classical theorem of
Pick, or more precisely its extension to matricial
data (for example, \cite{bgr}), the Nevanlinna-Pick problem with data
(\ref{NPdata}) is solvable if and only if the ``Pick matrix"
$$
\left[\frac{I - W^*_k W_\ell}{1 - \overline \lambda_k
\lambda_\ell}\right]^n_{k, \ell=1}
$$
is positive.

The following is our criterion for the solvability of an interpolation problem for $\hol(\d,\ov{\cale})$.  It relates the problem  to a family of classical matricial Nevanlinna-Pick problems.
\begin{theorem}\label{tetra_star2}
Let $\lambda_1,\dots,\lambda_n$ be distinct points in $\mathbb{D}$ and let $(x_1^k,x_2^k,x_3^k)\in\cale$ for $k=1,\dots,n.$ The following statements are equivalent. 
\begin{enumerate}
\item[(1)] There exists an analytic function $x:\mathbb{D}\to \ov{\cale}$ such that 
\begin{equation}\label{tetra-eqn1}
x(\lambda_k)=(x_1^k,x_2^k,x_3^k),\quad k=1,\dots,n;
\end{equation}
\item[(2)] there exist $b_k,c_k\in\mathbb{C}$ such that
\begin{equation}\label{tetra eqnsecond}
b_kc_k=x_1^kx_2^k-x_3^k,\qquad  k=1,\dots,n,
\end{equation} 
and the Nevanlinna-Pick problem with data
\begin{equation}\label{tetra eqn3}
\lambda_k\mapsto\left[\begin{array}{lr}
x_1^k&b_k\\c_k&x_2^k
\end{array}\right],\qquad  k=1,\dots,n,
\end{equation}
is solvable.
\end{enumerate}
\end{theorem}
\begin{proof}
$(1)\Rightarrow(2)$ Suppose there is an analytic function $x:\mathbb{D}\to \ov{\cale}$ such that equation \eqref{tetra-eqn1} holds.   By Proposition \ref{tetra-intBLY}, there is a function $F$ in the $2\times 2$ Schur class, that is, that $\| F \|_\infty \leq 1$ such that
\begin{equation}\label{tetra_property1}
x=(F_{11},\;F_{22},\;\det F).
\end{equation}
Let $b_k=F_{12}(\lambda_k)\;\;\text{and}\;\;c_k=F_{21}(\lambda_k),\; k=1,\dots,n.$
Then 
\[
F(\lambda_k)=\left[\begin{array}{lr}
x_1(\lambda_k)&F_{12}(\lambda_k)\\F_{21}(\lambda_k)&x_2(\lambda_k)
\end{array}\right]=\left[\begin{array}{lr}
x_1^k&b_k\\c_k&x_2^k
\end{array}\right],\] 
and so 
\[
x_3^k=x_3(\lambda_k)=x_1^kx_2^k-b_kc_k.
\] 
Thus 
\[
b_kc_k=x_1^kx_2^k-x_3^k, \quad  k=1,\dots,n.
\]
Hence the equations \eqref{tetra eqnsecond} are satisfied, and for this choice of $b_k,c_k$ the matricial Nevanlinna-Pick problem with the data \eqref{tetra eqn3} is solvable by $F.$    Condition (2) of the theorem is satisfied.

$(2)\Rightarrow(1)$ Suppose that $b_k,\;c_k$ exist such that the equations \eqref{tetra eqnsecond} hold and the Nevanlinna-Pick problem with data \eqref{tetra eqn3} is solvable by a function  $F\in \mathcal{S}^{2\times 2}$.    Let 
\[
x=(F_{11}, F_{22},\det F).
\]
By  Proposition \ref{cond_Ebar}, $x\in\hol(\d,\ov{\cale})$.

Since conditions \eqref{tetra eqn3} hold, for $k=1,\cdots,n,$
\begin{align*}
x_1(\lambda_k)&=F_{11}(\lambda_k)=x_1^k,\\
x_2(\lambda_k)&=F_{22}(\lambda_k)=x_2^k,\\
x_3(\lambda_k)&=\det F(\lambda_k)=x_1(\lambda_k)x_2(\lambda_k)-b_kc_k=x_3^k.
\end{align*}
Thus condition (1) of the theorem holds.
\end{proof}

As a consequence we obtain a solvability criterion for a finite interpolation problem for  $\hol(\d,\ov{\cale})$ in terms of the positivity of one of a family of matrices.
\begin{corollary} \label{easycor} Let $\lambda_1,\dots,\lambda_n$ be distinct points in $\mathbb{D}$ and let $(x_1^k,x_2^k,x_3^k)\in\cale$  such that $x_1^k x_2^k \neq x_3^k$ for $k=1,\dots, n$.
 The following three statements are equivalent.
 \begin{enumerate}
 \item[(1)] There exists  a function $x\in\hol(\d,\ov{\cale})$ such that
\[
x(\la_k)=  (x_1^k,x_2^k,x_3^k) \quad \mbox{ for }k=1,\dots, n;
\]
 \item[(2)] there exist $b_1,\dots,b_n$ and $c_1,\dots,c_n\in\mathbb{C}$ such that
$b_kc_k=x_1^kx_2^k-x_3^k$ for $k=1,\dots, n$ and
 \begin{equation}\label{matricial Pick}
 \left[\frac{I-\left[\begin{array}{lr}
 x_1^k&b_k\\c_k&x_2^k
 \end{array}\right]^*\left[\begin{array}{lr}
 x_1^\ell&b_\ell\\c_\ell&x_2^\ell
 \end{array}\right]}{1-\overline{\lambda_k}\lambda_\ell}\right]_{k,\ell=1}^n\geq 0.
 \end{equation}
\item[(3)] there exists a rational  $\cale$-inner function $x:\mathbb{D}\to \ov{\cale}$ such that 
\begin{equation}\label{tetra-eqn2}
x(\lambda_k)=(x_1^k,x_2^k,x_3^k),\quad  k=1,\dots,n.
\end{equation}
 \end{enumerate}
 \end{corollary}
The statement  $(1)\Leftrightarrow(2)$ follows immediately from Theorem \ref{tetra_star2} and Pick's theorem. The statement that  $(1)\Leftrightarrow(3)$ is contained in \cite[Theorem 8.1]{BLY}. 

\section{The structured singular value and the tetrablock}\label{mu}

The original motivation for the study of the tetablock in \cite{AWY07} was the wish to throw light on the ``problem of $\mu$-synthesis", which arises in the theory of robust control \cite{Do,doyle2,doylestein,DuPa}.   Operator theorists were greatly intrigued to learn around 1980 that some of their favourite theorems, such as those of Pick and Nehari, played a significant role in some problems of engineering design.  It was immediately a challenge to extend those theorems to provide an analysis of some of the more subtle optimization problems posed by engineers.  Up to now operator theorists have had limited success in meeting this challenge, and so it remains relevant to analyse test cases of these optimization problems.

The symbol $\mu$ is used to denote the {\em structured singular value} of a matrix relative to a space of linear transformations -- see \cite{DP,DuPa} for full definitions.  The structured singular value is a cost function which generalises the operator norm and is designed to reflect structural information about modelling uncertainty.  The usual operator norm of a matrix and the spectral radius of a square matrix are both instances of $\mu$.  Accordingly, two special cases of the $\mu$-synthesis problem are the classical Nevanlinna-Pick problem and its spectral variant, in which the operator norm is replaced by the spectral radius.  The tetrablock is associated with a third special case of $\mu$, which we denote by $\mudiag$.

\begin{definition}\label{defmudiag}
$\Diag$ denotes the space of diagonal $2\times 2$ matrices over $\c$.\\
For any $2\times 2$ matrix $A$,
\[
\mudiag(A) \df \left( \inf \{\|X\|: X\in\Diag, \, 1-AX  \mbox{ is singular}\}  \right)\inv.
\]
In the event that $1-AX$ is nonsingular for every $X\in\Diag$ we define $\mudiag(A)$ to be $0$.
\end{definition}
The tetrablock is connected to $\mudiag$ by the simple fact
 \cite[Theorem 9.1]{AWY07} that a $2\times 2$ matrix $A=[a_{ij}]$ satisfies $\mudiag(A) < 1$ if and only if $(a_{11}, a_{22}, \det A) \in \cale$.

The {\em   $\mudiag$-synthesis problem} is the following special case of $\mu$-synthesis. 
\vspace*{0.3cm}

{\em Given distinct points $\lambda_1, \dots, \lambda_n$ in
 $\mathbb{D}$ and $2\times 2$ matrices
$W_1,\dots, W_n$, where $ n \ge 1$, find conditions for the existence of an
analytic $2\times 2$ matrix-valued function $F$ on $\mathbb{D}$ such
that}
$$
F(\lambda_k) = W_k, \qquad k =1, 2, \dots, n,
$$
{\it and}
$$
\mudiag(F(\lambda)) \le 1 \quad  \mbox{ for all } \lambda \in
\mathbb{D}.
$$
There is a Matlab package \cite{M} for the numerical solution of this problem, but as yet no very efficient algorithm and little supporting theory.
The problem can be reduced to the finite interpolation problem for $\hol(\d,\ov{\cale})$, as is shown in the following result, which is \cite[Theorem 9.2]{AWY07}. 
\begin{proposition}\label{musynthequivfortetra} 
Let $\lambda_1,\dots,\lambda_n$ be distinct points in $\mathbb{D}$ and let 
\[
W_k=\begin{bmatrix}w_{ij}^k\end{bmatrix}_{i,j=1}^2 \quad \mbox{ for }k=1,\dots, n,
\]
  be $2\times 2$ matrices such that 
$ w_{12}^k w_{21}^k \neq  0$ for $k=1,\dots, n$.
The following statements are equivalent.
\begin{enumerate}
\item There exists an analytic function $F:\d\to\c^{2\times 2}$ such that 
\[
F(\la_k) = W_k \quad \mbox{ for } k=1,\dots, n
\]
and
\beq\label{mu-ineq}
\sup_{\la\in\D} \mu_{\mathrm{Diag}}(F(\lambda))  \leq 1;
\eeq
\item there exists an analytic function $\phi \in\hol(\D, \ov{\cale})$ such that
\beq\label{x-ineq}
\phi (\la_k) = (w_{11}^k,w_{22}^k,\det W_k) \quad \mbox{ for } k= 1,\dots,n.
\eeq
\end{enumerate}
\end{proposition}

On combining Proposition \ref{musynthequivfortetra}  with  Theorem \ref{tetra_star2} we obtain a criterion for the solvability of a
$\mudiag$-synthesis problem.
\begin{theorem} \label{main-theorem}
Let $\lambda_1,\dots,\lambda_n$ be distinct points in $\mathbb{D}$ and let 
\[
W_k=\begin{bmatrix}w_{ij}^k\end{bmatrix}_{i,j=1}^2 \quad \mbox{ for }k=1,\dots, n
\]
  be $2\times 2$ matrices such that 
$ w_{12}^k w_{21}^k \neq  0$   for $k=1,\dots, n$.
The following statements are equivalent.
\begin{enumerate}
\item There exists an analytic function $F:\d\to\c^{2\times 2}$ such that 
\[
F(\la_k) = W_k \quad \mbox{ for } k=1,\dots, n
\]
and
\beq\label{mu-ineq2}
\sup_{\la\in\D} \mudiag(F(\lambda))  \leq 1;
\eeq
 \item[(2)] there exist $b_1,\dots,b_n$ and $c_1,\dots,c_n\in\mathbb{C}$ such that
\[
b_kc_k=w_{11}^k w_{22}^k-\det W_k \quad \mbox{ for } k=1,\dots, n
\]
and
 \begin{equation}\label{matricial Pick condition}
 \left[\frac{I-\left[\begin{array}{lr}
 w_{11}^k&b_k\\c_k&w_{22}^k
 \end{array}\right]^*\left[\begin{array}{lr}
 w_{11}^\ell&b_\ell\\c_\ell&w_{22}^\ell
 \end{array}\right]}{1-\overline{\lambda_k}\lambda_\ell}\right]_{k,\ell=1}^n\geq 0.
 \end{equation}
 \end{enumerate}
 \end{theorem}
\begin{proof}
By Theorem \ref{musynthequivfortetra}, since  $w_{12}^k w_{21}^k \neq 0$ for $k=1,\dots, n,$ condition (1) is equivalent to the existence of $\varphi\in \hol(\mathbb{D},\ov{\cale})$ such that 
\[
\varphi(\lambda_k)= (w_{11}^k, w_{22}^k,\det W_k) \quad \mbox{ for } k=1,\dots, n.
\]  
Hence, by Theorem \ref{tetra_star2}, the matricial Nevanlinna-Pick problem with data 
\begin{equation*}
\lambda_k\mapsto\left[\begin{array}{lr}
w_{11}^k&b_k\\c_k&w_{22}^k
\end{array}\right],\qquad  k=1,\dots, n,
\end{equation*}
is solvable for some $b_k, c_k\in\mathbb{C}$ satisfying 
\begin{equation*}
b_kc_k=w_{11}^k w_{22}^k-\det W_k \quad \mbox{ for } k=1,\dots, n.
\end{equation*}
By the matricial version of Pick's theorem, this matricial Nevanlinna-Pick problem is solvable if and only if the Pick condition \eqref{matricial Pick condition} is satisfied.
\end{proof}


The proof suggests a way to construct solutions of a $\mu_{\rm Diag}$-synthesis  
 problem.  Suppose we are given $\lambda_1, \dots, \lambda_n,
W_1, \dots, W_n$ such that condition (2) of Theorem \ref{main-theorem} is satisfied,
and we can somehow find suitable numbers $b_1,\dots,b_n, c_1,\dots,c_n$ to
make the Pick matrix positive.  This is of course a nonconvex problem, and we do not know of a good algorithm to find the $b_k$ and $c_k$.
Once they are found, however, there are various ways of constructing functions $\chi$ in the Schur
class  such that
$$
\chi(\lambda_k) = \\
  \left[\begin{array}{ccc}
w_{11}^k &&b_k\\ \\
c_k&& w_{22}^k\end{array}\right], \qquad k=1,\dots,n,
$$
(see for example \cite{bgr}).  The function
$$
\varphi = (\chi_{11}, \chi_{22}, \det \chi)
$$
is then an analytic function from $\D$ to $\overline{\cale}$ such that
$$
\varphi(\lambda_k) = (w^k_{11}, w^k_{22},\det W_k), \qquad k=1,\dots,n,
$$
and the outline following Proposition \ref{tetra-intBLY} shows how to use $\varphi$ to
construct an analytic $2 \times 2$ matrix-valued function $F$ in $\D$ such
that $F(\lambda_k)= W_k$ and $\mudiag(F(\lambda)) \leq 1$ for all $\lambda \in \D$.\\


\begin{thebibliography}{50}
\bibitem{aak}
V. M.  Adamyan, D. Z.  Arov and M. G.  Krein, Analytic properties of
Schmidt pairs of a Hankel operator and generalized Schur--Takagi
problem, {\it Mat.  Sbornik} {\bf 86} (1971) 33--73. 

\bibitem{AWY07}
A. A. Abouhajar, M. C. White and N. J. Young,
A Schwarz lemma for a domain related to $\mu$-synthesis,
\emph{J. Geom. Anal.}, {\bf 17} (4) (2007) 717--750.


\bibitem{AY4} J. Agler and N. J. Young,  The two-point
spectral Nevanlinna--Pick problem,
{\em Integral Equ. Oper. Theory} {\bf 37} (2000) 375--385.

\bibitem{bb} J. A. Ball and V. Bolotnikov, Nevanlinna--Pick interpolation for Schur--Agler class functions on domains with matrix polynomial
defining function in $\C^n$, {\it New York J. Math.} {\bf 11} (2005)
247--290.


\bibitem{BallHuaman}  J. A. Ball and M. D. Guerra Huam\'an, Test functions, Schur-Agler classes and transfer-function realizations: the matrix-valued setting, {\em Complex  Anal. Oper. Theory} {\bf 7} (2003) 529-575.

\bibitem{BMV} J. A. Ball, G. Marx and V. Vinnikov, Interpolation and transfer-function realization for the noncommutative Schur-Agler class, arXiv 1602.00762 .

\bibitem{ballHorst1}  J. A. Ball and S. ter Horst, Robust control, multidimensional systems and multivariable Nevanlinna-Pick interpolation, In: Topics in Operator Theory, pp. 13--88, {\bf OT   203}  Birkh\"auser, Basel, 2010.

\bibitem{ballHorst2} J. A. Ball and S. ter Horst, Multivariable operator-valued Nevanlinna-Pick interpolation: a survey, In: Operator Algebras, Operator Theory and Applications, pp. 1--73, {\bf OT  195} Birkh\"auser, Basel, 2010.

\bibitem{bgr} J. A. Ball, I. C. Gohberg and L. Rodman,
{\em Interpolation of Rational Matrix Functions}, Operator Theory:
Advances and Applications {\bf 45},  Birkh\"auser, Basel, 1990.

\bibitem{bh}   J. A. Ball and J. W. Helton,  Interpolation problems of Pick-Nevanlinna and Loewner types for meromorphic matrix functions: parametrization of the set of all solutions. \emph{Integral Equ. Oper. Theory} \textbf{ 9} (1986) 155--203.

\bibitem{tetrabhatt}
T. Bhattacharyya, The tetrablock as a spectral set,
\emph{Indiana Univ. Math. J.} \textbf{63} (6) (2014) 1601-1629.


\bibitem{BLY} D. C. Brown, Z. A. Lykova and N. J. Young,
A rich structure related to the construction of holomorphic matrix functions,
  {\it J. Funct. Anal.} {\bf 272} (2017) 1704--1754. 


\bibitem{Do} J. C. Doyle, Analysis of feedback systems
with structured  uncertainties. {\em Proc. IEE-D} {\bf 129} (1982) 242--250.

\bibitem{DP} J. C. Doyle and A. Packard, The complex structured
singular value, {\em Automatica J. IFAC}, {\bf 29} (1993) 71-109.

\bibitem{doyle2}
J. C. Doyle,
Structured uncertainty in control system design,
\emph{24th IEEE Conference on Decision and Control}, {\bf 24} (1985) 260-265.

\bibitem{doylestein}
J. C. Doyle and G. Stein,
Multivariable feedback design: concepts for a classical/modern synthesis,
\emph{IEEE Transactions on Automatic Control}, {\bf 26} (1981)  4-16.


\bibitem{DM}  M. A. Dritschel and S. McCullough, Test functions, kernels, realizations  and interpolation, in {\em Operator Theory, Structured Matrices and Dilations:  T. Constantinescu Memorial Volume} (ed. M. Bakonyi, A. Gheondea, M. Putinar and J. Rovnyak), pp. 153-179, Theta Series in Advanced Mathematics, Bucharest, 2007.

\bibitem{DuPa}
G. E. Dullerud and F. G. Paganini,
\emph{A Course in Robust Control Theory: A Convex Approach},
Springer, 2000. 

\bibitem{edizwontetra}
A. Edigarian, L. Kosi\'{n}ski and W. Zwonek,
The Lempert theorem and the tetrablock,
\emph{J. Geom. Anal.} \textbf{23} (4) (2013) 1818-1831.



\bibitem{jp}
M. Jarnicki and P. Pflug, {\em Invariant Distances and Metrics in Complex Analysis},  2nd Extended Edition, De Gruyter, Berlin, 2013.

\bibitem{kos} 
L. Kosi\'{n}ski and W. Zwonek,
Nevanlinna-Pick problem and uniqueness of left inverses in convex domains, symmetrized bidisc and tetrablock, {\em J. Geom.  Anal.} {\bf 26} (2016)  1863-1890.

\bibitem{M} {\em Matlab} $\mu$-{\em Analysis and Synthesis
Toolbox},  The Math Works
Inc.,Natick, Massachusetts (http://www.mathworks.com/products/muanalysis/).


\bibitem{NF} B. Sz-Nagy and C. Foias, {\em Harmonic Analysis of
Operators on Hilbert Space}, Akad\' emiai Kiad\' o,
Budapest, 1968.


\bibitem{NTT} N. Nikolov, P. J. Thomas and  M. Trybula, Gromov (non)hyperbolicity of certain domains in $\c^2$,
{\em Forum Math. } {\bf 28} (2016) 783--794.


\bibitem{pal} S. Pal,  The failure of rational dilation on the tetrablock, {\em J. Funct. Anal.}
{\bf 269} (2015) 1903-1924.

\bibitem{Pick}  
G.  Pick, \"Uber die Beschr\"ankungen analytischer Funktionen, welche
durch vorgegebene Funktionswerte bewirkt werden, {\it Math.  Ann.}  {\bf
77} (1916) 7--23.



\bibitem{sar67}
D. Sarason, Generalised interpolation in $H^\infty$,
\emph{Trans. Amer. Math. Soc.} \textbf{127} (2) (1967) 179-203.


\bibitem{You08} N. J. Young, The automorphism group of the tetrablock, {\em J. London Math. Soc.} (2)  {\bf 77} (2008) 757--770.

\end{thebibliography}
\end{document}